\documentclass[11pt,leqno]{article}

\usepackage{amsfonts,amsmath,amssymb,amsthm}
\usepackage{fullpage}
\usepackage{dsfont}
\usepackage{hyperref}
\hypersetup{hidelinks}

\newtheorem{theorem}{Theorem}[section]
\newtheorem{lemma}[theorem]{Lemma}
\newtheorem{proposition}[theorem]{Proposition}
\newtheorem{corollary}[theorem]{Corollary}
\numberwithin{equation}{section}

\newcommand{\ls}{\leqslant}
\newcommand{\gr}{\geqslant}
\newcommand{\E}{\mathbb{E}}
\newcommand{\Pp}{\mathbb{P}}
\newcommand{\R}{\mathbb{R}}
\newcommand{\supp}{\operatorname{supp}}
\newcommand{\1}{\mathds{1}}
\newcommand{\norm}[1]{\left\lVert#1\right\rVert}

\usepackage{setspace}
\begin{document}
\small

\title{\bf Online balancing of vectors with small coordinates}
\author{Antonios Hmadi}
\date{}
\maketitle

\begin{abstract}\footnotesize
Let $v_1,\ldots,v_T\in B_2^m$ be fixed in advance and revealed sequentially, and assume that $\norm{v_t}_\infty\ls d^{-1/2}$ for some $d\gr1$ and every $1\ls t\ls T$.
There are absolute constants $L,C,c>0$ and a randomized online signing such that
$$ \Pp\left\{ \max_{k\ls T}\norm{\sum_{t=1}^k\varepsilon_t v_t}_\infty>6L \right\} \ls CT\exp\!\left[-\frac{cd}{\ln^2(ed)}\right]. $$
Consequently, constant prefix discrepancy holds with probability at least $1-\varepsilon$ once
$d$ is at least $C\ln\frac{3T}{\varepsilon}\left[\ln\left(e+\ln\frac{3T}{\varepsilon}\right)\right]^2$.
In particular, every fixed sequence of vectors $a_t\in[-1,1]^m$ with at most $d$ nonzero coordinates admits an online signing with prefix discrepancy $O(\sqrt d)$ and failure probability at most $CT\exp[-cd/\ln^2(ed)]$.
We also prove a nonuniform version in which the failure probability depends on the individual parameters $d_t=\norm{v_t}_\infty^{-2}$, and a lower bound showing that a universal constant prefix discrepancy is impossible when $d=o(\ln T)$.
We identify the corresponding $\ln^2 d$ barrier for the compact-potential method and extend the argument to general symmetric target bodies admitting a quadratic smoothness estimate.
\end{abstract}

%%%%%%%%%%%%%%%%%%%%%%%%%%%%%%%%%%%%%%%%%%%%%%%%%%%%%%%%%%%%%%%%%%%%%%%%%%%%%%%%%%%%%%%%%%%%%%%%%%%%%%%%%%%%%%%%%%%%%%%%%%%%%%%%%%%%%%
\section{Introduction}
%%%%%%%%%%%%%%%%%%%%%%%%%%%%%%%%%%%%%%%%%%%%%%%%%%%%%%%%%%%%%%%%%%%%%%%%%%%%%%%%%%%%%%%%%%%%%%%%%%%%%%%%%%%%%%%%%%%%%%%%%%%%%%%%%%%%%%

Vector balancing asks how efficiently a family of vectors can cancel after assigning signs to them.
In the offline problem all vectors are known in advance, whereas in the online problem they are revealed one by one and each sign must be chosen before the future arrivals are seen.
Let $v_1,\ldots,v_T\in\R^m$ be fixed vectors revealed sequentially and let $\varepsilon_t\in\{-1,1\}$ be chosen after observing $v_t$.
We study the prefix discrepancy
$$ \max_{1\ls k\ls T}\norm{\sum_{t=1}^k\varepsilon_t v_t}_\infty. $$

The offline Koml\'os problem asks for an absolute terminal discrepancy when $\norm{v_t}_2\ls1$.
The classical general estimate is $O(\sqrt{\ln(eT)})$, as follows from Banaszczyk's Gaussian-measure theorem~\cite{Banaszczyk}, and Bansal and Jiang obtained a $\widetilde O(\ln^{1/4}(eT))$ estimate~\cite{BansalJiang}.

For the online problem under the same Euclidean hypothesis, Alweiss, Liu and Sawhney obtained an $O(\ln(mT))$ prefix-discrepancy bound in linear time~\cite{ALS}.
Kulkarni, Reis and Rothvoss proved that every prefix can be made $O(1)$-sub-Gaussian, which gives the optimal $O(\sqrt{\ln(eT)})$ bound, and established a matching lower bound against an oblivious adversary~\cite{KRR}.
Aden-Ali gave an $O(mT)$-time algorithm with the same upper bound~\cite{AdenAli}.

For $v\neq0$ put $d(v)=\norm v_\infty^{-2}$ and set $d(0)=\infty$.
Thus $d(v)\gr d$ means $\norm v_\infty\ls d^{-1/2}$.
We call $v$ $d$-sparse if it has at most $d$ nonzero coordinates.
We ask whether the two conditions $\norm v_2\ls1$ and $d(v)\gr d$, without any restriction on the number of nonzero coordinates of $v$, permit constant prefix discrepancy.

Altschuler and Tikhomirov~\cite{AT} study the related online Beck--Fiala problem, where the arrivals $a_t\in[-1,1]^m$ are $d$-sparse.
For every fixed $\eta>0$, their main theorem gives prefix discrepancy $O_\eta(\sqrt d)$ with failure probability
$$ C_\eta T\exp\!\left[-\frac{c_\eta d}{\ln^{2+\eta}(ed)}\right], $$
and the sufficient range $d\gr C_\eta(\ln T)(\ln\ln T)^{2+\eta}$.
After normalization by $\sqrt d$, the vectors $v_t=a_t/\sqrt d$ satisfy $\norm{v_t}_2\ls1$ and $\norm{v_t}_\infty\ls d^{-1/2}. $

We use the compact Metropolis framework of~\cite[Section~2]{AT}: the transition, the product law, the notion of a balanced state and the curvature criterion.
For Theorem~\ref{thm:main-online} we also use the bump density and the one-dimensional estimates in~\cite[Lemmas~3.1--3.2]{AT}.
The three-way coupling producing genuine signs is from~\cite[Lemma~6]{AdenAli}.
The result below applies to the full small coordinate class $B_2^m\cap d^{-1/2}B_\infty^m$, with no restriction on the support of the arrivals. 
On the $d$-sparse subclass it also replaces the loss $\ln^{2+\eta}(ed)$ in~\cite{AT} by the critical scale $\ln^2(ed)$.

Our principal result is the following.

\begin{theorem}\label{thm:critical-uniform}
There are absolute constants $L,C,c>0$ and a randomized online algorithm, allowed to depend on the common parameter $d$, with the following property.
For every $m,T\gr1$, every $d\gr1$, and every fixed sequence $v_1,\ldots,v_T\in\R^m$ satisfying
$$ \norm{v_t}_2\ls1, \qquad \norm{v_t}_\infty\ls d^{-1/2} \qquad(1\ls t\ls T), $$
the algorithm produces signs $\varepsilon_t\in\{-1,1\}$ such that
\begin{equation}\label{eq:critical-uniform-failure}
\Pp\left\{ \max_{k\ls T}\norm{\sum_{t=1}^k\varepsilon_t v_t}_\infty>6L \right\} \ls C T \exp\!\left[-\frac{c d}{\ln^2(ed)}\right].
\end{equation}
Consequently, for every $0<\varepsilon<1$, the sufficient condition
\begin{equation}\label{eq:critical-uniform-threshold}
d\gr C\ln\frac{3T}{\varepsilon} \left[\ln\left(e+\ln\frac{3T}{\varepsilon}\right)\right]^2
\end{equation}
makes the probability in \eqref{eq:critical-uniform-failure} at most $\varepsilon$.
In particular, $d\gr C(\ln T)[\ln\ln(e^eT)]^2$ gives constant prefix discrepancy with probability at least $1-T^{-c}$.
\end{theorem}

The algorithm in Theorem~\ref{thm:critical-uniform} uses exact sampling from the one-dimensional law constructed below and exact evaluation of its potential in the Metropolis ratios.
We do not claim a finite-bit running-time bound for approximate sampling and evaluation.
The obstruction $d=o(\ln T)$ in Proposition~\ref{prop:flattened-lower} is obtained by tensorizing the two-dimensional lower bound of Kulkarni, Reis and Rothvoss~\cite{KRR}.

The next statement retains the individual values $d_t=d(v_t)$.

\begin{theorem}\label{thm:main-online}
For every $\beta>0$ there are constants $L_\beta,C_\beta,c_\beta>0$ and a randomized online algorithm with the following property.
For every $m,T\gr1$ and every fixed sequence $v_1,\ldots,v_T\in\R^m$ satisfying $\norm{v_t}_2\ls1$, set $d_t=d(v_t)$.
Then the algorithm produces signs $\varepsilon_t\in\{-1,1\}$ such that
$$ \Pp\left\{ \max_{k\ls T}\norm{\sum_{t=1}^k\varepsilon_t v_t}_\infty>6L_\beta \right\} \ls C_\beta\sum_{t=1}^T \exp\!\left[-\frac{c_\beta d_t} {\ln^{2+2/\beta}(e d_t)}\right], $$
where the summand corresponding to $d_t=\infty$ is interpreted as zero.
\end{theorem}

The sparse discrepancy problem goes back to Beck and Fiala~\cite{BeckFiala}.
For constructive forms of Banaszczyk's theorem and related discrepancy walks we refer to \cite{BansalDadushGarg,GramSchmidtWalk,LovettMeka,LSS,SmirnovVershynin}; see also the survey \cite{BansalSurvey}.
Online vector balancing under stochastic arrivals and sparse random-input models is considered in \cite{BJSS,BJMSS,BansalSpencer,ATThreshold}.
These input models are different from the fixed deterministic sequence considered here.

Although~\cite[Proposition~3.5]{AT} is stated under the normalized small coordinate hypotheses above, its curvature estimate already uses the weighted relations $v_i^2\ls d^{-1}$ and $\sum_i v_i^2\ls1$, whereas the treatment of the boundary event uses that the arrival has at most $d$ nonzero coordinates.
For general vectors we instead put $r_i(v)=v_i^{-2}$ and use the identity
$$ \sum_{i:v_i\neq0}\frac1{r_i(v)}=\norm v_2^2\ls1. $$
Together with the unequal scale concentration estimate in Section~\ref{sec:concentration} this proves Theorem~\ref{thm:main-online} without any restriction on the number of nonzero coordinates.
For Theorem~\ref{thm:critical-uniform} we instead choose the invariant law depending on the common parameter $d$. 
The point is that the potential need only have the critical curvature behavior up to the finite scale $d$; beyond this scale it may grow faster in order to remain compactly supported. 
This finite scale construction leads to the loss $\ln^2(ed)$.

The two results on compact one-dimensional potentials in Section~\ref{sec:offline-consequence} show that this is the critical scale for the pointwise curvature method used here. 
A fixed compact potential requires a logarithmic exponent strictly larger than $2$, whereas for a potential depending on $d$ the corresponding finite-scale curvature parameter is necessarily of order at least $\ln^2 d$.
Thus the construction of Section~\ref{sec:critical-bump} attains this finite-scale barrier up to absolute constants.

Finally, we give a separate geometric extension to general symmetric
target bodies.
General target norms have been considered for stochastic online arrivals in~\cite{BJMSS}, while Smirnov and Vershynin~\cite{SmirnovVershynin} use a Metropolis walk in a general convex body with the possibility of discarding steps.
Our setting is different, the arrivals are fixed in advance, every arrival receives a sign and both the size of the arrivals and the discrepancy are measured in the norm of $K$.
If $K$ is uniformly equivalent to the unit ball of a norm satisfying a quadratic smoothness estimate with parameter $\sigma$, Theorem~\ref{thm:smooth-body} gives constant $K$-prefix discrepancy under the condition
$$\norm{v_t}_K\ls d^{-1/2}, \qquad d\gr C A^2\sigma\left(m+\ln\frac{3T}{\varepsilon}\right).$$
To the best of our knowledge, this fixed-sequence formulation in the intrinsic norm of the target body does not appear explicitly in the literature.

The paper is organized as follows.
Section~\ref{sec:preliminaries} collects the notation, terminology and auxiliary results used throughout the paper.
Section~\ref{sec:critical-bump} constructs the invariant law used in Theorem~\ref{thm:critical-uniform} and proves the theorem.
Section~\ref{sec:concentration} proves the concentration estimate used in Theorem~\ref{thm:main-online}.
Section~\ref{sec:compact} proves the estimate for a single vector and Theorem~\ref{thm:main-online}.
Section~\ref{sec:offline-consequence} contains the lower bound and the two results on compact one-dimensional potentials.
Section~\ref{app:geometric} proves the geometric extension for general target bodies satisfying a quadratic smoothness condition.

%%%%%%%%%%%%%%%%%%%%%%%%%%%%%%%%%%%%%%%%%%%%%%%%%%%%%%%%%%%%%%%%%%%%%%%%%%%%%%%%%%%%%%%%%%%%%%%%%%%%%%%%%%%%%%%%%%%%%%%%%%%%%%%%%%%%%%
\section{Notation and auxiliary results}\label{sec:preliminaries}
%%%%%%%%%%%%%%%%%%%%%%%%%%%%%%%%%%%%%%%%%%%%%%%%%%%%%%%%%%%%%%%%%%%%%%%%%%%%%%%%%%%%%%%%%%%%%%%%%%%%%%%%%%%%%%%%%%%%%%%%%%%%%%%%%%%%%%

For $m\gr1$, we write $B_2^m:=\{x\in\R^m:\norm{x}_2\ls1\}$ and $B_\infty^m:=\{x\in\R^m:\norm{x}_\infty\ls1\}$.
For $x=(x_1,\ldots,x_m)\in\R^m$, we write $\norm{x}_2=\left(\sum_{i=1}^m x_i^2\right)^{1/2}$, $\norm{x}_\infty=\max_{1\ls i\ls m}|x_i|$, and $\supp x=\{i:x_i\neq0\}$.
We write $\1_E$ for the indicator of an event or set $E$, and $\E$, $\Pp$ for expectation and probability.
All logarithms are natural.
The letters $c,C,c_1,C_1,\ldots$ denote positive absolute constants whose values may change from line to line; a subscript, as in $C_\beta$, indicates allowed dependence on that parameter.
We write $A\asymp B$ when both $A\ls CB$ and $B\ls CA$ hold, and $\propto$ for equality up to a positive multiplicative constant.

An \emph{online signing} of fixed vectors $v_1,\ldots,v_T$ chooses $\varepsilon_t\in\{-1,1\}$ after seeing $v_t$ and before seeing $v_{t+1},\ldots,v_T$.
The sequence is chosen by an \emph{oblivious adversary} when it is fixed before the private randomness of the algorithm is sampled.
The quantity $\max_{1\ls k\ls T}\norm{\sum_{t=1}^k\varepsilon_t v_t}_\infty$ is the \emph{prefix discrepancy}; the same norm with $k=T$ only is the terminal discrepancy.
For $v\neq0$ put $d(v):=\norm{v}_\infty^{-2}$, and set $d(0):=\infty$.
Thus $d(v)\gr d$ is exactly the coordinatewise bound $\norm v_\infty\ls d^{-1/2}$.
For a sequence $(v_t)$ we write $d_t=d(v_t)$.
This condition imposes no restriction on the number of nonzero coordinates of $v$.

For $i\in\supp v$ put
\begin{equation}\label{eq:coordinate-scales}
r_i(v):=v_i^{-2}.
\end{equation}
Then
\begin{equation}\label{eq:coordinate-scale-relations}
r_i(v)\gr d(v), \qquad \sum_{i\in\supp v}\frac1{r_i(v)}=\norm v_2^2.
\end{equation}
When the vector $v$ is clear from the context we simply write $r_i$.

For $q>0$ and $x\gr1$ put
\begin{equation}\label{eq:ellq}
\ell_q(x):=\ln^q(ex).
\end{equation}

We shall use the following notation for one-dimensional compact densities.
Let $\nu$ be a probability measure on $(-1,1)$ with positive density
\begin{equation}\label{eq:generic-density}
\rho(y)=Z^{-1}e^{-\varphi(y)}\1_{\{|y|<1\}}, \qquad \varphi\in C^2((-1,1)),
\end{equation}
where $Z$ is the normalizing constant, and let $Y$ have law $\nu$.
We refer to $\varphi$ as the potential of the density.
For $b\gr1$ and $r\gr1$ put
\begin{equation}\label{eq:general-boundary}
B_r:=\{y\in(-1,1):1-|y|\ls b r^{-1/2}\},
\end{equation}
and, for $y\notin B_r$,
\begin{equation}\label{eq:general-envelope}
G_r(y):=1+\sup_{|u-y|\ls r^{-1/2}}(\varphi''(u))_+,
\end{equation}
with $G_r(y)=0$ on $B_r$.

We next recall the terminology of Altschuler and Tikhomirov~\cite[Section~2.1 and Definitions~2.3--2.4]{AT}.
The Metropolis transition used below depends only on the target density.
Let $K$ be a symmetric convex body and let $f$ be a positive density on $\operatorname{int}K$, extended by zero outside $K$.
For $x\in\operatorname{int}K$ and $v\in\R^m$ define
\begin{align}
 p_v^+(x)&=\frac13\min\left\{1,\frac{f(x+v)}{f(x)}\right\}, \label{eq:metropolis-plus}\\
 p_v^-(x)&=\frac13\min\left\{1,\frac{f(x-v)}{f(x)}\right\},\\
 p_v^0(x)&=1-p_v^+(x)-p_v^-(x).
 \label{eq:metropolis-zero}
\end{align}
The factors $\min\{1,f(x\pm v)/f(x)\}$ are the usual Metropolis acceptance factors: they make the transition reversible with respect to $f$ through detailed balance.
The step is called \emph{lazy} because it has a holding probability $p_v^0(x)$; the factor $1/3$ gives $p_v^\pm(x)\ls1/3$ and hence $p_v^0(x)\gr1/3$.
Following~\cite{AT}, $x$ is called \emph{$v$-balanced} when $p_v^+(x)+p_v^-(x)\gr\frac13$.
The word ``balanced'' is imported from~\cite{AT}; this condition is exactly what is needed in Lemma~\ref{lem:three-way-coupling} to couple three increments taking values in $\{-1,0,1\}$ so that their sum is always $\pm1$.
If, on $\operatorname{int}K$, the density has the form $f\propto e^{-U}$, then for $x,x\pm v\in\operatorname{int}K$ we write
\begin{equation}\label{eq:central-difference}
D_v(x):=U(x+v)+U(x-v)-2U(x)
\end{equation}
for the central second difference of $U$ in the direction $v$.
The next lemma is \cite[Lemma~2.2]{AT}; we include the proof for completeness.

\begin{lemma}\label{lem:metropolis-stationarity}
Let $X\sim\mu$, where $\mu$ has density $f$, and conditionally on $X=x$ let $X'=x+\delta v$, with $\delta\in\{-1,0,1\}$ having the probabilities \eqref{eq:metropolis-plus}--\eqref{eq:metropolis-zero}.
Then $X'\sim\mu$.
Consequently, for every deterministic sequence $(v_t)$, the inhomogeneous walk obtained by applying the $v_t$-transition at time $t$ has marginal law $\mu$ at every time, provided it starts from $\mu$.
\end{lemma}

\begin{proof}
The two directed moves across every edge satisfy detailed balance:
$$ f(x)p_v^+(x)=\frac13\min\{f(x),f(x+v)\}=f(x+v)p_v^-(x+v). $$
Indeed, writing $f'$ for the density after one step and interpreting all terms outside $K$ as zero, detailed balance gives
$$ f'(y) =f(y)p_v^0(y)+f(y-v)p_v^+(y-v)+f(y+v)p_v^-(y+v) =f(y)\bigl(p_v^0(y)+p_v^-(y)+p_v^+(y)\bigr)=f(y). $$
The assertion for a deterministic sequence follows by induction.
\end{proof}

The following criterion is \cite[Lemma~2.5]{AT}; again we include the short proof.

\begin{lemma}\label{lem:curvature-criterion}
Let $f\propto e^{-U}$ on the interior of its convex support.
If $x$ and $x\pm v$ belong to this interior and the quantity $D_v(x)$ from \eqref{eq:central-difference} satisfies $D_v(x)\ls2\ln2$, then $x$ is $v$-balanced.
\end{lemma}

\begin{proof}
Put $\Delta_\pm:=U(x\pm v)-U(x)$.
If one of $\Delta_+,\Delta_-$ is nonpositive, then the corresponding transition probability equals $1/3$.
Otherwise both are positive and
$$ p_v^+(x)+p_v^-(x) =\frac13\bigl(e^{-\Delta_+}+e^{-\Delta_-}\bigr) \gr\frac23\exp\left(-\frac{\Delta_++\Delta_-}{2}\right) \gr\frac13, $$
because $\Delta_++\Delta_-=D_v(x)\ls2\ln2$.
\end{proof}

We shall use Lemma~6 of Aden-Ali~\cite{AdenAli}.
It couples three random variables taking values in $\{-1,0,1\}$ so that their sum is always a genuine sign $\pm1$.

\begin{lemma}\label{lem:three-way-coupling}
For $j\in\{1,2,3\}$, let $a_j,b_j\in[0,1/3]$ satisfy $a_j+b_j\gr1/3$.
There is a coupling $(\delta_1,\delta_2,\delta_3)$ of random variables taking values in $\{-1,0,1\}$ such that
$$ \Pp\{\delta_j=1\}=a_j, \qquad \Pp\{\delta_j=-1\}=b_j, \qquad \Pp\{\delta_j=0\}=1-a_j-b_j $$
for every $j$, and
\begin{equation}\label{eq:three-way-sign}
\delta_1+\delta_2+\delta_3\in\{-1,1\} \qquad\hbox{almost surely}.
\end{equation}
\end{lemma}

We shall also use the following arrival-dependent form of~\cite[Proposition~2.7]{AT}, in which the individual failure probabilities are retained.

\begin{proposition}\label{prop:nonuniform-triplet}
Let $K\subset\R^m$ be a symmetric convex body, let $\mu$ be a probability law with density $f>0$ on $\operatorname{int}K$, extended by zero outside $K$, and let $v_1,\ldots,v_T\in\R^m$ be fixed in advance.
Put $\theta_t:=\Pp_{X\sim\mu}\{X\hbox{ is not }v_t\hbox{-balanced}\}$.
Then there is a randomized online signing such that
$$ \Pp\left\{ \max_{k\ls T}\norm{\sum_{t=1}^k\varepsilon_t v_t}_K>6 \right\} \ls3\sum_{t=1}^T\theta_t. $$
\end{proposition}

\begin{proof}
Start three auxiliary states $X_{0,1},X_{0,2},X_{0,3}$ with marginal law $\mu$; they may initially be sampled independently.
At time $t$, as long as no failure has occurred and all three states $X_{t-1,j}$ are $v_t$-balanced, apply Lemma~\ref{lem:three-way-coupling} with $a_j=p_{v_t}^+(X_{t-1,j})$ and $b_j=p_{v_t}^-(X_{t-1,j})$.
Set
\begin{equation}\label{eq:triplet-update}
X_{t,j}=X_{t-1,j}+\delta_{t,j}v_t, \qquad \varepsilon_t:=\delta_{t,1}+\delta_{t,2}+\delta_{t,3}.
\end{equation}
By \eqref{eq:three-way-sign}, $\varepsilon_t\in\{-1,1\}$.
If at least one state is not $v_t$-balanced, declare failure, output an arbitrary sign, and evolve each auxiliary state according to its own transition probabilities \eqref{eq:metropolis-plus}--\eqref{eq:metropolis-zero}; after failure the three walks may be evolved independently.
This convention is only used to keep their marginal laws defined on the entire probability space.

At every successful step, the coupling in Lemma~\ref{lem:three-way-coupling} preserves the prescribed one-walk conditional marginal of each $\delta_{t,j}$.
The same is true by construction after failure.
Lemma~\ref{lem:metropolis-stationarity} therefore gives $X_{t,j}\sim\mu$ for $0\ls t\ls T$ and $1\ls j\ls3$.
Since the sequence $(v_t)$ is fixed independently of the private states, $\Pp\{X_{t-1,j}\hbox{ is not }v_t\hbox{-balanced}\}=\theta_t$.
Consequently, a union bound gives
\begin{equation}\label{eq:triplet-failure-union}
\Pp\{\hbox{failure by time }T\} \ls\sum_{t=1}^T\sum_{j=1}^3 \Pp\{X_{t-1,j}\hbox{ is not }v_t\hbox{-balanced}\} =3\sum_{t=1}^T\theta_t.
\end{equation}

On the complementary event, summing \eqref{eq:triplet-update} first over $j$ and then over $t\ls k$ yields $\sum_{t=1}^k\varepsilon_t v_t=\sum_{j=1}^3(X_{k,j}-X_{0,j})$.
Every auxiliary state lies in $K$.
Since $K$ is symmetric and convex, $X_{k,j}-X_{0,j}\in2K$, and therefore $\norm{\sum_{t=1}^k\varepsilon_t v_t}_K\ls6$ for every prefix $k$.
Combining this with \eqref{eq:triplet-failure-union} proves the proposition.
\end{proof}

For $\beta>0$, consider the compactly supported density introduced by Altschuler and Tikhomirov~\cite[Section~3]{AT}.
Set $q_\beta:=2+\frac{2}{\beta}$, $S(y):=1-y^2$, $A(y):=S(y)^{-\beta}$, and $\Phi_\beta(y):=e^{A(y)}$,
and define
\begin{equation}\label{eq:bump-density}
h_\beta(y):=Z_\beta^{-1}e^{-\Phi_\beta(y)}\1_{\{|y|<1\}},
\end{equation}
where $Z_\beta$ is the normalizing constant.
Extending $h_\beta$ by zero outside $(-1,1)$ gives a $C^\infty$ compactly supported function which vanishes to infinite order at the endpoints; this is the reason for the term ``bump density'' used in~\cite{AT}.
For this density we take $\varphi=\Phi_\beta$, $b=6$ in \eqref{eq:general-boundary}, and use the corresponding $G_r$ from \eqref{eq:general-envelope}.

The next lemma is the form of \cite[Lemmas~3.1--3.2]{AT} that we need for every real $r\gr1$; the proof is included to make this dependence explicit.

\begin{lemma}\label{lem:bump-input}
There are constants $C_\beta,c_\beta>0$ such that, for every real $r\gr1$,
\begin{align}
 \Pp\{Y\in B_r\} &\ls C_\beta\exp\{-\exp(c_\beta r^{\beta/2})\}, \label{eq:bump-boundary-tail}\\
 \Pp\{G_r(Y)>x\} &\ls C_\beta\exp\!\left[-\frac{c_\beta x}{\ell_{q_\beta}(x)}\right], \qquad1\ls x\ls r. \label{eq:bump-curvature-tail}
\end{align}
\end{lemma}

\begin{proof}
All constants in this proof may depend on $\beta$.
We first record the pointwise curvature estimate
\begin{equation}\label{eq:bump-pointwise-curvature}
0<\Phi_\beta''(y)\ls C e^{A(y)}A(y)^{q_\beta}, \qquad -1<y<1.
\end{equation}
Indeed,
$$ A'(y)=2\beta yS(y)^{-\beta-1}, \qquad A''(y)=2\beta S(y)^{-\beta-1} +4\beta(\beta+1)y^2S(y)^{-\beta-2}. $$
Consequently,
$$ \Phi_\beta''(y) =e^{A(y)}\bigl((A'(y))^2+A''(y)\bigr), \qquad (A'(y))^2+A''(y)\ls C A(y)^{2+2/\beta} =C A(y)^{q_\beta}. $$
Here the last estimate uses $A(y)\gr1$ and $S(y)^{-1}=A(y)^{1/\beta}$, and proves \eqref{eq:bump-pointwise-curvature}.

We next prove the boundary estimate.
For $y\in B_r$, $S(y)=(1-|y|)(1+|y|)\ls12r^{-1/2}$, and therefore $A(y)\gr12^{-\beta}r^{\beta/2}$ and $\Phi_\beta(y)\gr\exp(c r^{\beta/2})$.
Since $B_r$ has total length at most two, integrating the preceding pointwise bound gives
$$ \Pp\{Y\in B_r\} \ls2Z_\beta^{-1}\exp\{-\exp(c r^{\beta/2})\} $$
for every $r\gr1$, proving \eqref{eq:bump-boundary-tail}.

We next prove the curvature tail.
By increasing the leading constant, it suffices to treat $r$ and $x$ larger than constants depending on $\beta$.
Fix $y\notin B_r$ and $u$ with $|u-y|\ls r^{-1/2}$.
Since $|S(u)-S(y)|=|u-y|\,|u+y|\ls2r^{-1/2}\ls\frac13S(y)$, we have $S(u)\gr\frac23S(y)$; the same lower bound holds for every point of the segment joining $y$ and $u$.
The mean value theorem consequently gives
$$ |A(u)-A(y)| \ls C r^{-1/2}S(y)^{-\beta-1} =C r^{-1/2}A(y)^{1+1/\beta}. $$

Choose a constant $a_0>0$, to be fixed below, and suppose that
\begin{equation}\label{eq:bump-small-potential}
e^{A(y)}\ls a_0\frac{x}{\ell_{q_\beta}(x)}.
\end{equation}
After decreasing $a_0$ we may assume $a_0\ls1$.
Since $x\ls r$, condition \eqref{eq:bump-small-potential} implies $A(y)\ls\ln x\ls\ln r$.
Thus the right-hand side of the preceding estimate tends to zero as $r\to\infty$.
Hence, for all sufficiently large $r$, uniformly in $1\ls x\ls r$, we have $e^{A(u)}\ls2e^{A(y)}$ and $A(u)\ls2A(y)$.
Using \eqref{eq:bump-pointwise-curvature} and \eqref{eq:bump-small-potential}, we obtain
$$ \Phi_\beta''(u) \ls C e^{A(y)}A(y)^{q_\beta} \ls C a_0\frac{x}{\ln^{q_\beta}(ex)}(\ln x)^{q_\beta} \ls C a_0x. $$
Choose $a_0$ so that the last expression is at most $x/2$.
For $x\gr2$ this gives $G_r(y)\ls1+x/2\ls x$.
Hence $G_r(y)>x$ implies $e^{A(y)}>a_0x/\ell_{q_\beta}(x)$ for all sufficiently large $r$ and $2\ls x\ls r$.

Finally, for every $s\gr1$, the definition of the density gives $\Pp\{e^{A(Y)}>s\}=Z_\beta^{-1}\int_{\{e^{A(y)}>s\}}e^{-e^{A(y)}}\,dy\ls2Z_\beta^{-1}e^{-s}$.
Combining the last two estimates gives
$$ \Pp\{G_r(Y)>x\} \ls C\exp\left[-\frac{a_0 x}{\ell_{q_\beta}(x)}\right] $$
for the unbounded parameter range.
Enlarging $C$ covers $1\ls x<2$ and the bounded range of $r$, completing the proof of \eqref{eq:bump-curvature-tail}.
\end{proof}

%%%%%%%%%%%%%%%%%%%%%%%%%%%%%%%%%%%%%%%%%%%%%%%%%%%%%%%%%%%%%%%%%%%%%%%%%%%%%%%%%%%%%%%%%%%%%%%%%%%%%%%%%%%%%%%%%%%%%%%%%%%%%%%%%%%%%%
\section{Proof of the uniform theorem}\label{sec:critical-bump}
%%%%%%%%%%%%%%%%%%%%%%%%%%%%%%%%%%%%%%%%%%%%%%%%%%%%%%%%%%%%%%%%%%%%%%%%%%%%%%%%%%%%%%%%%%%%%%%%%%%%%%%%%%%%%%%%%%%%%%%%%%%%%%%%%%%%%%

The density used for Theorem~\ref{thm:main-online} is fixed for all scales.
For Theorem~\ref{thm:critical-uniform}, the common lower bound $d$ is known in advance, and we choose the invariant density depending on $d$.

Put $\ell_d:=\ln(ed)$ and $\Lambda_d:=\ell_d^2$.
For $d$ larger than an absolute constant consider the nondecreasing function
\begin{equation}\label{eq:critical-profile}
H_d^0(u):= \begin{cases} u^2,&1\ls u\ls\ell_d,\\[1mm] \Lambda_d,&\ell_d\ls u\ls d,\\[1mm] \Lambda_d\left(\dfrac{\ln(eu)}{\ell_d}\right)^3,&u\gr d. \end{cases}
\end{equation}
We use the following elementary regularization.

\begin{lemma}
There is a smooth nondecreasing function $H_d:[1,\infty)\to(0,\infty)$ and an absolute constant $C$ such that
\begin{equation}\label{eq:critical-profile-regularity}
C^{-1}H_d^0(u)\ls H_d(u)\ls C H_d^0(u), \qquad 0\ls uH_d'(u)\ls C H_d(u), \qquad H_d(2u)\ls C H_d(u)
\end{equation}
for every $u\gr1$.
\end{lemma}

\begin{proof}
Put $g_d^0(s):=\ln H_d^0(e^s)$ for $s\gr0$, and extend $g_d^0$ to $(-\infty,0)$ by $g_d^0(s)=2s$.
The values of the three formulas in \eqref{eq:critical-profile} agree at both junctions, so $g_d^0$ is continuous and nondecreasing.
At every point of differentiability, $0\ls(g_d^0)'(s)\ls3$.
Indeed, the derivative is $2$ in the first regime, $0$ in the second, and $3/(1+s)$ in the third.
Thus $g_d^0$ is globally $3$-Lipschitz.

Fix a nonnegative $C^\infty$ function $\eta$ supported on $[-1/10,1/10]$ with integral one, set $g_d:=g_d^0*\eta$, and put $H_d(u):=\exp(g_d(\ln u))$ for $u\gr1$.
Convolution preserves monotonicity, while the Lipschitz bound gives $0\ls g_d'(s)\ls3$ and $|g_d(s)-g_d^0(s)|\ls3/10$.
Consequently $H_d\asymp H_d^0$ and $0\ls uH_d'(u)=g_d'(\ln u)H_d(u)\ls3H_d(u)$.
Finally, $\ln\frac{H_d(2u)}{H_d(u)}=g_d(\ln u+\ln2)-g_d(\ln u)\ls3\ln2$, which proves the doubling estimate with an absolute constant.
\end{proof}

Any fixed exponent strictly larger than $2$ may be used in the last regime.

The following integral is bounded by an absolute constant:
\begin{equation}\label{eq:critical-compactness}
 \int_1^\infty\frac{du}{u\sqrt{H_d(u)}} \ls C\left( \int_1^{\ell_d}\frac{du}{u^2} +\frac1{\ell_d}\int_{\ell_d}^{d}\frac{du}{u} +\frac1{\ell_d}\int_d^\infty \frac{du}{u(\ln(eu)/\ell_d)^{3/2}} \right) \ls C. 
\end{equation}
The same integral, starting from $2$, is bounded below by an absolute positive constant.

\begin{lemma}\label{lem:critical-potential}
There are absolute constants $c_0,C_0>0$ such that, for every sufficiently large $d$, there is an even convex potential $\varphi_d\in C^2((-1,1))$ satisfying $1\ls\varphi_d<\infty$ and $\lim_{|y|\uparrow1}\varphi_d(y)=\infty$, with the following properties.
Put $W_d(y):=1+\varphi_d(y)H_d(\varphi_d(y))$.
Then, for every $y\in(-1,1)$ and every $h\in\R$ satisfying
\begin{equation}\label{eq:critical-local-smallness}
h^2W_d(y)\ls c_0,
\end{equation}
one has $y\pm h\in(-1,1)$ and
\begin{equation}\label{eq:critical-local-second-difference}
\varphi_d(y+h)+\varphi_d(y-h)-2\varphi_d(y) \ls C_0h^2W_d(y).
\end{equation}
Moreover, if $Y_d$ has density
\begin{equation}\label{eq:critical-density}
\rho_d(y):=Z_d^{-1}e^{-\varphi_d(y)}\1_{\{|y|<1\}},
\end{equation}
then
\begin{align}
 \Pp\{W_d(Y_d)>x\} &\ls C\exp(-c x^{1/3}), &&1\ls x\ls\ell_d^3, \label{eq:critical-small-tail}\\
 \Pp\{W_d(Y_d)>x\} &\ls C\exp\left(-\frac{cx}{\Lambda_d}\right), &&\ell_d^3\ls x\ls d, \label{eq:critical-middle-tail}\\
 \E\left[W_d(Y_d)\1_{\{W_d(Y_d)>d\}}\right] &\ls C\exp\left(-\frac{cd}{\Lambda_d}\right). \label{eq:critical-tail-moment}
\end{align}
All constants are absolute.
\end{lemma}

\begin{proof}
Let $a_d:=\int_1^\infty\frac{du}{u\sqrt{H_d(u)}}$.
By \eqref{eq:critical-compactness} and the corresponding lower bound,
\begin{equation}\label{eq:critical-width-bounds}
c\ls a_d\ls C.
\end{equation}
For $|y|<1$ define $\varphi_d(y)$ implicitly by
\begin{equation}\label{eq:critical-implicit-potential}
\int_1^{\varphi_d(y)} \frac{du}{u\sqrt{H_d(u)}} =a_d y^2.
\end{equation}
This defines an even function with $\varphi_d(0)=1$ and $\varphi_d(y)\to\infty$ as $|y|\uparrow1$.
Away from the origin, differentiating gives
\begin{equation}\label{eq:critical-potential-ode}
\varphi_d'(y) =2a_d y\,\varphi_d(y)\sqrt{H_d(\varphi_d(y))}.
\end{equation}
To justify the assertion at the origin explicitly, as $u\downarrow1$ we have
$$ \frac1{u\sqrt{H_d(u)}} = \frac1{\sqrt{H_d(1)}} - \left( \frac1{\sqrt{H_d(1)}}+ \frac{H_d'(1)}{2H_d(1)^{3/2}} \right)(u-1) +O((u-1)^2). $$
Integrating this expansion and then applying the result with $s=\varphi_d(y)$ in \eqref{eq:critical-implicit-potential} give
$$ \int_1^s\frac{du}{u\sqrt{H_d(u)}}=\frac{s-1}{\sqrt{H_d(1)}}+O((s-1)^2), \qquad \varphi_d(y)=1+a_d\sqrt{H_d(1)}\,y^2+O(y^4). $$
Hence $\varphi_d$ is $C^2$ at the origin and $\varphi_d'(0)=0$.
If $q_d(u):=u\sqrt{H_d(u)}$, then \eqref{eq:critical-profile-regularity} and a second differentiation of \eqref{eq:critical-potential-ode} give
$$ 0\ls q_d'(u)=\sqrt{H_d(u)}+\frac{uH_d'(u)}{2\sqrt{H_d(u)}}\ls C\sqrt{H_d(u)}, \qquad \varphi_d''(y)=2a_d q_d(\varphi_d(y))+2a_d yq_d'(\varphi_d(y))\varphi_d'(y). $$
Using again \eqref{eq:critical-potential-ode}, the preceding estimate, $|y|<1$ and $H_d(u)\gr c$, we get
\begin{equation}\label{eq:critical-global-curvature}
0\ls\varphi_d''(y) \ls C\varphi_d(y)H_d(\varphi_d(y)) \ls C W_d(y) \qquad(|y|<1). \end{equation}
Thus $\varphi_d$ is convex.
Moreover, \eqref{eq:critical-profile}--\eqref{eq:critical-profile-regularity} give $\int_1^2du/(u\sqrt{H_d(u)})\gr c$.
By \eqref{eq:critical-width-bounds}, we may choose an absolute $\delta>0$ such that $a_d\delta^2<c$.
Then \eqref{eq:critical-implicit-potential} and the monotonicity of the integral imply $\varphi_d(y)\ls2$ whenever $|y|\ls\delta$.
Since $\varphi_d\gr1$ on $(-1,1)$, it follows that
\begin{equation}\label{eq:critical-normalization}
c\ls Z_d\ls C.
\end{equation}

We prove the local assertion.
By \eqref{eq:critical-profile-regularity}, the function $K_d(u):=1+uH_d(u)$ satisfies $K_d(2u)\ls CK_d(u)$.
Fix $\theta\in\{-1,1\}$ and let $\tau$ be the first $s\in[0,|h|]$ for which either $y+\theta s\notin(-1,1)$ or $\varphi_d(y+\theta s)=2\varphi_d(y)$; if neither event occurs, put $\tau=|h|$.
Before $\tau$, \eqref{eq:critical-potential-ode}, the monotonicity and doubling of $H_d$, and $|y+\theta s|<1$ give
$$ \left|\frac{d}{ds}\ln\varphi_d(y+\theta s)\right|\ls C\sqrt{H_d(\varphi_d(y))}\ls C\sqrt{W_d(y)}, \qquad \left|\ln\frac{\varphi_d(y+\theta\tau)}{\varphi_d(y)}\right|\ls C|h|\sqrt{W_d(y)}. $$
Choose $c_0$ in \eqref{eq:critical-local-smallness} so that the last quantity is smaller than $\ln2$.
The equality $\varphi_d(y+\theta\tau)=2\varphi_d(y)$ is then impossible.
Nor can the boundary be reached first, since continuity and the divergence of $\varphi_d$ at the boundary would force the level $2\varphi_d(y)$ to be crossed earlier.
Hence $\tau=|h|$ for both signs and $\sup_{|s|\ls|h|}\varphi_d(y+s)\ls2\varphi_d(y)$.
The doubling property of $K_d$ now gives $\sup_{|s|\ls|h|}W_d(y+s)\ls C W_d(y)$.
Combining the preceding estimate with \eqref{eq:critical-global-curvature} and the identity
$$ \varphi_d(y+h)+\varphi_d(y-h)-2\varphi_d(y) =\int_{-|h|}^{|h|}(|h|-|s|)\varphi_d''(y+s)\,ds $$
gives \eqref{eq:critical-local-second-difference}.

We next turn to the tails.
From \eqref{eq:critical-normalization},
\begin{equation}\label{eq:critical-potential-tail}
\Pp\{\varphi_d(Y_d)>u\}\ls Ce^{-u}, \qquad u\gr1.
\end{equation}
For $1\ls u\ls\ell_d$, \eqref{eq:critical-profile}--\eqref{eq:critical-profile-regularity} give $W_d\ls C(1+u^3)$, whereas for $\ell_d\ls u\ls d$ they give $W_d\ls C(1+\Lambda_du)$.
We verify the two inversions, including the junction ranges.
Let $u=\varphi_d(y)$.
If $1\ls x\ls\ell_d^3$ and $W_d(y)>x$, then either $u\ls\ell_d$, in which case $x<C(1+u^3)$, or $u>\ell_d$, in which case $u\gr\ell_d\gr x^{1/3}$.
After absorbing the bounded range of $x$, both cases give
$$ W_d(y)>x \quad\Longrightarrow\quad \varphi_d(y)>c x^{1/3} \qquad(1\ls x\ls\ell_d^3). $$

Now let $\ell_d^3\ls x\ls d$.
If $\ell_d\ls u\ls d$, then $x<C(1+\Lambda_du)$ implies $u>cx/\Lambda_d$.
If $u>d$, the same conclusion is automatic.
It remains to consider $u<\ell_d$.
Here $x<C(1+u^3)\ls C\ell_d^3$, so necessarily $x\asymp\ell_d^3$ up to an absolute factor; moreover $u>cx^{1/3}\asymp\ell_d$, and hence again $u>c'x/\Lambda_d$.
Thus, after changing absolute constants, $W_d(y)>x$ implies $\varphi_d(y)>cx/\Lambda_d$.
Combining the two implications with \eqref{eq:critical-potential-tail} proves \eqref{eq:critical-small-tail} and \eqref{eq:critical-middle-tail}.

Finally, for sufficiently large $d$, $W_d(y)>d$ implies $\varphi_d(y)\gr cd/\Lambda_d$.
Indeed, if $u\ls\ell_d$, then $W_d(y)\ls C(1+\ell_d^3)<d$ for large $d$; if $\ell_d\ls u\ls d$, the estimate $W_d(y)\ls C(1+\Lambda_du)$ gives $u\gr cd/\Lambda_d$; and if $u>d$ the conclusion is immediate.
Moreover, uniformly in $d$ and $u\gr1$, $1+uH_d(u)\ls C(1+u^6)\ls Ce^{u/2}$.
In the last regime we used $\Lambda_d\ls d^2\ls u^2$ and $\ln(eu)/\ell_d\ls u$; the other two regimes are immediate.
Therefore, using also \eqref{eq:critical-normalization} and the fact that the support has length two,
$$ \begin{aligned} \E\left[W_d(Y_d)\1_{\{W_d(Y_d)>d\}}\right] &\ls C\int_{\{\varphi_d\gr cd/\Lambda_d\}} W_d(y)e^{-\varphi_d(y)}\,dy\\ &\ls C\int_{\{\varphi_d\gr cd/\Lambda_d\}} e^{-\varphi_d(y)/2}\,dy \ls C\exp\left(-\frac{cd}{\Lambda_d}\right), \end{aligned} $$
which proves \eqref{eq:critical-tail-moment}.
\end{proof}

Notice that \eqref{eq:critical-profile} and \eqref{eq:critical-global-curvature} imply $\varphi_d''(y)\ls C\Lambda_d\varphi_d(y)$ whenever $\varphi_d(y)\ls d$.
Thus the construction attains the finite-scale lower bound of Proposition~\ref{prop:finite-scale-barrier}, up to absolute constants.
The last part of the definition of $H_d$ is used only above this range.

The next estimate controls the part above the common truncation level $d$ in expectation rather than by a union bound over the coordinates.

\begin{lemma}\label{lem:critical-weighted-concentration}
There are absolute constants $K,C,c>0$ such that the following holds.
Let $d$ be sufficiently large, let $W_1,\ldots,W_m$ be independent copies of $W_d(Y_d)$, and let $w_1,\ldots,w_m\gr0$ satisfy
\begin{equation}\label{eq:critical-weights}
\max_iw_i\ls\frac1d, \qquad \sum_iw_i\ls1.
\end{equation}
Then
$$ \Pp\left\{\sum_{i=1}^m w_iW_i>K\right\} \ls C\exp\left(-\frac{cd}{\Lambda_d}\right). $$
\end{lemma}

\begin{proof}
Put $\overline W_i:=\min\{W_i,d\}$.
We first claim that there are absolute $\kappa,C_0>0$ such that
\begin{equation}\label{eq:critical-truncated-mgf}
\E e^{\theta\overline W_i}\ls e^{C_0\theta}, \qquad 0\ls\theta\ls\frac{\kappa}{\Lambda_d}.
\end{equation}
For such $\theta$, $\E e^{\theta\overline W_i}=1+\theta\int_0^d e^{\theta x}\Pp\{W_i>x\}\,dx$.
Let $x_0:=\ell_d^3$.
For large $d$, $x_0\ls d$.
Since $x/\Lambda_d\ls x^{1/3}$ on $[1,x_0]$, we have $\theta x\ls\kappa x/\Lambda_d\ls\kappa x^{1/3}$.
Choose $\kappa$ smaller than half the constant in \eqref{eq:critical-small-tail}.
Then, after decreasing $\kappa$ once more for the second estimate,
$$ \begin{aligned} \int_1^{x_0}e^{\theta x}\Pp\{W_i>x\}\,dx &\ls C\int_1^\infty e^{-c x^{1/3}/2}\,dx\ls C,\\ \int_{x_0}^{d}e^{\theta x}\Pp\{W_i>x\}\,dx &\ls C\int_{x_0}^{\infty}\exp\left(-\frac{cx}{2\Lambda_d}\right)\,dx\ls C\Lambda_de^{-c\ell_d/2}\ls C. \end{aligned} $$
Here $x_0/\Lambda_d=\ell_d$, and $\Lambda_de^{-c\ell_d/2}=\ln^2(ed)(ed)^{-c/2}$ is uniformly bounded after the bounded range of $d$ is absorbed into the constant.
The interval $[0,1]$ contributes at most an absolute constant.
We have therefore shown $\E e^{\theta\overline W_i}\ls1+C\theta\ls e^{C_0\theta}$, which proves \eqref{eq:critical-truncated-mgf}.

Take $\lambda:=\frac{\kappa d}{\Lambda_d}$.
By \eqref{eq:critical-weights}, $\lambda w_i\ls\kappa/\Lambda_d$.
Independence and \eqref{eq:critical-truncated-mgf} yield
$$ \E\exp\left(\lambda\sum_iw_i\overline W_i\right) \ls \exp\left(C_0\lambda\sum_iw_i\right) \ls e^{C_0\lambda}. $$
Hence, for an absolute $K_0>C_0+1$,
\begin{equation}\label{eq:critical-truncated-sum}
\Pp\left\{\sum_iw_i\overline W_i>K_0\right\} \ls\exp\left(-\frac{cd}{\Lambda_d}\right).
\end{equation}

For the part above $d$, \eqref{eq:critical-tail-moment}, \eqref{eq:critical-weights} and Markov's inequality give
$$ \begin{aligned} \E\sum_iw_iW_i\1_{\{W_i>d\}} &\ls\left(\sum_iw_i\right)\E\left[W_d(Y_d)\1_{\{W_d(Y_d)>d\}}\right]\ls C\exp\left(-\frac{cd}{\Lambda_d}\right),\\ \Pp\left\{\sum_iw_iW_i\1_{\{W_i>d\}}>1\right\} &\ls C\exp\left(-\frac{cd}{\Lambda_d}\right). \end{aligned} $$
Combining this with \eqref{eq:critical-truncated-sum} proves the lemma with $K=K_0+1$.
\end{proof}

We can now prove the required one-step estimate.

\begin{proposition}\label{prop:critical-balance}
There are absolute constants $L,C,c>0$ such that, for every sufficiently large $d$ and every dimension $m$, the product law $X=L(Y_{d,1},\ldots,Y_{d,m})$, where $Y_{d,1},\ldots,Y_{d,m}$ are independent with density \eqref{eq:critical-density}, satisfies
\begin{equation}\label{eq:critical-balance-tail}
\Pp\{X\hbox{ is not }v\hbox{-balanced}\} \ls C\exp\left(-\frac{cd}{\Lambda_d}\right)
\end{equation}
for every $v\in\R^m$ with $\norm v_2\ls1$ and $\norm v_\infty\ls d^{-1/2}$.
\end{proposition}

\begin{proof}
Put $w_i=v_i^2$.
Then $\max_iw_i\ls\frac1d$ and $\sum_iw_i=\norm v_2^2\ls1$.
Let $W_i=W_d(Y_{d,i})$.
By Lemma~\ref{lem:critical-weighted-concentration}, outside an event of probability at most the right-hand side of \eqref{eq:critical-balance-tail}, $\sum_iw_iW_i\ls K$, where $K$ is absolute.

Write $h_i=v_i/L$.
On this event, $h_i^2W_i\ls\frac{K}{L^2}$ for every $i$.
Choose the absolute scaling constant $L$ so large that $K/L^2\ls c_0$, with $c_0$ from Lemma~\ref{lem:critical-potential}.
Then $Y_{d,i}\pm h_i\in(-1,1)$ and, for the product potential of $X$,
$$ \varphi_d(Y_{d,i}+h_i)+\varphi_d(Y_{d,i}-h_i)-2\varphi_d(Y_{d,i})\ls C_0h_i^2W_i, \qquad D_v(X)\ls\frac{C_0}{L^2}\sum_iw_iW_i\ls\frac{C_0K}{L^2}. $$
Increasing $L$ once more, we may assume that the last quantity is at most $2\ln2$.
Lemma~\ref{lem:curvature-criterion} implies that $X$ is $v$-balanced.
This proves \eqref{eq:critical-balance-tail}.
\end{proof}

\begin{proof}[Proof of Theorem~\ref{thm:critical-uniform}]
Assume first that $d$ exceeds the absolute threshold in Proposition~\ref{prop:critical-balance}, and use its product law as the invariant measure in Proposition~\ref{prop:nonuniform-triplet}.
For every arrival $v_t$, $\theta_t\ls C\exp\left(-\frac{cd}{\Lambda_d}\right)$.
Since the support of the product law is $L B_\infty^m$, Proposition~\ref{prop:nonuniform-triplet} gives
$$ \Pp\left\{ \max_{k\ls T}\norm{\sum_{t=1}^k\varepsilon_tv_t}_\infty>6L \right\} \ls CT\exp\left(-\frac{cd}{\ln^2(ed)}\right), $$
which is \eqref{eq:critical-uniform-failure}.
The bounded range of $d$ is absorbed by increasing the leading constant $C$, in which case the probability estimate is vacuous.

For \eqref{eq:critical-uniform-threshold}, put $H:=\ln\frac{3T}{\varepsilon}$ and $M:=\ln(e+H)$.
If $d_0=C_1HM^2$, then $\ln(ed_0)\ls C_2M$ with an absolute $C_2$, provided $C_1$ is larger than an absolute constant.
Hence $\frac{d_0}{\ln^2(ed_0)}\gr cC_1H$.
The function $s\mapsto s/\ln^2(es)$ is increasing beyond an absolute constant.
Thus the same lower bound, with changed constants, holds for every $d\gr d_0$.
Choosing $C_1$ sufficiently large makes the right-hand side of \eqref{eq:critical-uniform-failure} at most $\varepsilon$.
The final high-probability statement follows by taking $\varepsilon=T^{-c}$ and adjusting the constants.
\end{proof}

\begin{corollary}\label{cor:critical-beck-fiala}
There are absolute constants $L,C,c>0$ such that the following holds.
Let $a_1,\ldots,a_T\in[-1,1]^m$ be fixed in advance and assume that every $a_t$ has at most $d$ nonzero coordinates.
Then there is a randomized online signing for which
$$ \Pp\left\{ \max_{k\ls T}\norm{\sum_{t=1}^k\varepsilon_t a_t}_\infty>6L\sqrt d \right\} \ls CT\exp\!\left[-\frac{cd}{\ln^2(ed)}\right]. $$
In particular, $d\gr C(\ln T)[\ln\ln(e^eT)]^2$ yields prefix discrepancy $O(\sqrt d)$ with probability at least $1-T^{-c}$.
\end{corollary}

\begin{proof}
Apply Theorem~\ref{thm:critical-uniform} to $v_t=a_t/\sqrt d$.
Then $\norm{v_t}_2\ls1$ and $\norm{v_t}_\infty\ls d^{-1/2}$.
\end{proof}

Taking the sparsity parameter equal to the ambient dimension gives the following consequence.

\begin{corollary}

There are absolute constants $L,C,c>0$ such that, for every fixed sequence $a_1,\ldots,a_T\in[-1,1]^m$, there is a randomized online signing satisfying
$$ \Pp\left\{ \max_{k\ls T}\norm{\sum_{t=1}^k\varepsilon_ta_t}_\infty >6L\sqrt m \right\} \ls CT\exp\!\left[-\frac{cm}{\ln^2(em)}\right]. $$
Consequently, if $T\ls\exp\!\left[\frac{c m}{2\ln^2(em)}\right]$, then this probability is at most $C\exp[-c m/(2\ln^2(em))]$.
\end{corollary}

\begin{proof}
Apply Corollary~\ref{cor:critical-beck-fiala} with $d=m$.
\end{proof}

%%%%%%%%%%%%%%%%%%%%%%%%%%%%%%%%%%%%%%%%%%%%%%%%%%%%%%%%%%%%%%%%%%%%%%%%%%%%%%%%%%%%%%%%%%%%%%%%%%%%%%%%%%%%%%%%%%%%%%%%%%%%%%%%%%%%%%
\section{A concentration estimate}\label{sec:concentration}
%%%%%%%%%%%%%%%%%%%%%%%%%%%%%%%%%%%%%%%%%%%%%%%%%%%%%%%%%%%%%%%%%%%%%%%%%%%%%%%%%%%%%%%%%%%%%%%%%%%%%%%%%%%%%%%%%%%%%%%%%%%%%%%%%%%%%%
We first record two elementary estimates.

\begin{lemma}\label{lem:calculus}
Let $q,a,c>0$.
For all sufficiently large $r$, the functions
$$ \frac{\ell_q(r)}{r}, \qquad r\exp\!\left[-\frac{cr}{\ell_q(r)}\right], \qquad r\exp\{-\exp(c r^a)\} $$
are nonincreasing, and, after enlarging the threshold if necessary,
$$ r\exp\!\left[-\frac{cr}{\ell_q(r)}\right] \ls \exp\!\left[-\frac{cr}{2\ell_q(r)}\right], \qquad r\exp\{-\exp(c r^a)\} \ls \exp\!\left[-\frac{cr}{\ell_q(r)}\right]. $$
\end{lemma}

\begin{proof}
Put $H(r)=\ln(er)$.
Differentiating $\ell_q(r)/r$ and $r/\ell_q(r)$ gives
$$ \left(\frac{\ell_q(r)}r\right)'=\frac{\ell_q(r)}{r^2}\left(\frac q{H(r)}-1\right), \qquad \left(\frac r{\ell_q(r)}\right)'=\frac1{\ell_q(r)}\left(1-\frac q{H(r)}\right). $$
The asserted monotonicities follow for large $r$; the two final estimates follow from $\ln r=o(r/\ell_q(r))$ and $\ln r+r/\ell_q(r)=o(e^{cr^a})$.
\end{proof}

\begin{lemma}\label{lem:aggregation}
Fix $q,A,a>0$.
There are constants $C,c>0$ such that, whenever $r_i\gr d\gr1$, $\sum_i r_i^{-1}\ls1$, and $\Pp(E_i)\ls A\exp\!\left[-\frac{a r_i}{\ell_q(r_i)}\right]$, one has
$$ \Pp\left(\bigcup_iE_i\right) \ls C\exp\!\left[-\frac{c d}{\ell_q(d)}\right]. $$
\end{lemma}

\begin{proof}
For large $d$, the union bound and Lemma~\ref{lem:calculus} give
$$ \Pp\left(\bigcup_iE_i\right) \ls A\sum_i\frac1{r_i} \left(r_i\exp\!\left[-\frac{a r_i}{\ell_q(r_i)}\right]\right) \ls A d\exp\!\left[-\frac{a d}{\ell_q(d)}\right] \ls C\exp\!\left[-\frac{c d}{\ell_q(d)}\right]. $$
The bounded range of $d$ is absorbed into $C$.
\end{proof}

The next theorem extends the almost-exponential average-tail argument in \cite[Lemma~3.3]{AT} to unequal truncation scales; the number of coordinates is replaced by the condition in \eqref{eq:abstract-scales}.

\begin{theorem}\label{thm:abstract-concentration}
Fix $q,A,a>0$, and let $\ell_q$ be as in \eqref{eq:ellq}.
There are constants $K,C,c>0$, depending only on $q,A,a$, with the following property.
Let $I$ be a finite index set, let $d\gr1$, and let $(r_i)_{i\in I}$ satisfy
\begin{equation}\label{eq:abstract-scales}
r_i\gr d, \qquad \sum_{i\in I}\frac1{r_i}\ls1.
\end{equation}
Let $(Z_i)_{i\in I}$ be independent nonnegative random variables such that
\begin{equation}\label{eq:abstract-tail}
\Pp\{Z_i>x\} \ls A\exp\!\left[-\frac{a x}{\ell_q(x)}\right], \qquad 1\ls x\ls r_i.
\end{equation}
Then
$$ \Pp\left\{ \sum_{i\in I}\frac{Z_i}{r_i}>K \right\} \ls C\exp\!\left[-\frac{c d}{\ell_q(d)}\right]. $$
\end{theorem}

\begin{proof}
Let $r_0$ be the maximum of the finitely many thresholds supplied by Lemma~\ref{lem:calculus} for the constants used below.
By increasing $C$, the conclusion is automatic for $1\ls d<r_0$, so assume $d\gr r_0$.
Put $\overline Z_i:=\min\{Z_i,r_i\}$.
First, \eqref{eq:abstract-tail} at $x=r_i$, the second condition in \eqref{eq:abstract-scales} and Lemma~\ref{lem:calculus} give
\begin{equation} \label{eq:abstract-truncation}
 \sum_{i\in I}\Pp\{Z_i>r_i\} \ls A\sum_{i\in I}\frac1{r_i}   \left(r_i\exp\!\left[-\frac{a r_i}{\ell_q(r_i)}\right]\right) \ls A d\exp\!\left[-\frac{a d}{\ell_q(d)}\right] \ls C\exp\!\left[-\frac{c d}{\ell_q(d)}\right].
\end{equation}

Choose $0<\kappa\ls\min\{1,a/2\}$ and put $\lambda:=\frac{\kappa d}{\ell_q(d)}$ and $\theta_i:=\frac{\lambda}{r_i}$.
Since $r\mapsto\ell_q(r)/r$ is nonincreasing on $[d,\infty)$, $\theta_i=\frac{\kappa d}{r_i\ell_q(d)}\ls\frac{\kappa}{\ell_q(r_i)}$.
For a nonnegative random variable $W\ls r_i$ almost surely, $\E e^{\theta_iW}=1+\theta_i\int_0^{r_i}e^{\theta_i x}\Pp\{W>x\}\,dx$.
Apply this identity to $W=\overline Z_i$.
On $[0,1]$ we use the trivial probability bound.
For $1\ls x\ls r_i$, monotonicity of $\ell_q$ and the preceding bound imply $\theta_i\ls\frac{\kappa}{\ell_q(r_i)}\ls\frac{\kappa}{\ell_q(x)}\ls\frac{a}{2\ell_q(x)}$.
Consequently,
$$ \E e^{\theta_i\overline Z_i}\ls1+\theta_i\left(e-1+A\int_1^\infty\exp\!\left[-\frac{a x}{2\ell_q(x)}\right]dx\right)\ls\exp(C_0\theta_i), $$
where the improper integral is finite because $x/\ell_q(x)$ dominates $\ln x$ as $x\to\infty$.

Independence and \eqref{eq:abstract-scales} now yield
$$ \E\exp\left(\lambda\sum_{i\in I}\frac{\overline Z_i}{r_i}\right) =\prod_{i\in I}\E e^{\theta_i\overline Z_i} \ls\exp\left(C_0\lambda\sum_{i\in I}\frac1{r_i}\right) \ls e^{C_0\lambda}.  $$
Choosing $K>C_0+1$ and applying Chernoff's inequality gives
$$ \Pp\left\{ \sum_{i\in I}\frac{\overline Z_i}{r_i}>K \right\} \ls\exp[-(K-C_0)\lambda] \ls\exp\!\left[-\frac{c d}{\ell_q(d)}\right]. $$
Combining this estimate with \eqref{eq:abstract-truncation} proves the theorem.
\end{proof}

The number of variables enters only through the second condition in \eqref{eq:abstract-scales}.

%%%%%%%%%%%%%%%%%%%%%%%%%%%%%%%%%%%%%%%%%%%%%%%%%%%%%%%%%%%%%%%%%%%%%%%%%%%%%%%%%%%%%%%%%%%%%%%%%%%%%%%%%%%%%%%%%%%%%%%%%%%%%%%%%%%%%%
\section{The Metropolis construction}\label{sec:compact}
%%%%%%%%%%%%%%%%%%%%%%%%%%%%%%%%%%%%%%%%%%%%%%%%%%%%%%%%%%%%%%%%%%%%%%%%%%%%%%%%%%%%%%%%%%%%%%%%%%%%%%%%%%%%%%%%%%%%%%%%%%%%%%%%%%%%%%

Let $\nu$, $Y$, $B_r$ and $G_r$ be as in Section~\ref{sec:preliminaries}.
Assume that, for some $q,A,a>0$ and $b\gr1$,
\begin{align}
 \Pp\{Y\in B_r\} &\ls A\exp\!\left[-\frac{a r}{\ell_q(r)}\right], \qquad r\gr1, \label{eq:boundary-tail-assumption}\\
 \Pp\{G_r(Y)>x\} &\ls A\exp\!\left[-\frac{a x}{\ell_q(x)}\right], \qquad 1\ls x\ls r. \label{eq:curvature-tail-assumption}
\end{align}

Let $Y_1,\ldots,Y_m$ be independent copies of $Y$ and let $L\gr1$.
We write $X=L(Y_1,\ldots,Y_m)$ and $Q_L=(-L,L)^m$, and denote by $\mu$ the law of $X$ and by $f$ its density.
Up to an additive constant in the exponent, $f(x)\propto e^{-U(x)}\1_{Q_L}(x)$, where $U(x)=\sum_{i=1}^m\varphi(x_i/L)$.

The next theorem replaces the support-size step in~\cite[Proposition~3.5]{AT}.

\begin{theorem}\label{thm:general-balance}
Assume that the measure $\nu$ in \eqref{eq:generic-density} satisfies \eqref{eq:boundary-tail-assumption}--\eqref{eq:curvature-tail-assumption} for some $q,A,a>0$ and $b\gr1$.
There are constants $L,C,c>0$, depending only on $q,A,a,b$, such that the following holds in every dimension $m$.
For every $v\in\R^m$ with $\norm{v}_2\ls1$, one has
\begin{equation}\label{eq:general-balance-tail}
\Pp\{X\hbox{ is not }v\hbox{-balanced}\} \ls C\exp\!\left[-\frac{c d(v)}{\ell_q(d(v))}\right],
\end{equation}
with the right-hand side interpreted as zero when $v=0$.
\end{theorem}

The bound depends on $v$ only through $d(v)$ and is independent of $|\supp v|$.

\begin{proof}
For $v=0$ the assertion is trivial.
Assume that $v\neq0$.
Let $I=\supp v$ and put $r_i=r_i(v)$ as in \eqref{eq:coordinate-scales}.
By \eqref{eq:coordinate-scale-relations}, $r_i\gr d(v)$ and $\sum_{i\in I}r_i^{-1}=\norm{v}_2^2\ls1$.
By Lemma~\ref{lem:aggregation} and \eqref{eq:boundary-tail-assumption}, outside an event of probability at most
\begin{equation}\label{eq:general-boundary-failure}
C\exp\!\left[-\frac{c d(v)}{\ell_q(d(v))}\right],
\end{equation}
we have $Y_i\notin B_{r_i}$ for every $i\in I$.

Set $z_i=v_i/L$.
Since $b\gr1$, $L\gr1$, and $Y_i\notin B_{r_i}$, we have $|z_i|\ls|v_i|=r_i^{-1/2}<1-|Y_i|$, so $Y_i\pm z_i\in(-1,1)$.
The integral central-difference formula gives
$$ \begin{aligned} \varphi(Y_i+z_i)+\varphi(Y_i-z_i)-2\varphi(Y_i) &=z_i^2\int_{-1}^1(1-|s|)\varphi''(Y_i+s z_i)\,ds\ls z_i^2G_{r_i}(Y_i),\\ D_v(X)&\ls\frac1{L^2}\sum_{i\in I}\frac{G_{r_i}(Y_i)}{r_i}. \end{aligned} $$
The random variables $G_{r_i}(Y_i)$ are independent and satisfy the hypotheses of Theorem~\ref{thm:abstract-concentration} by \eqref{eq:curvature-tail-assumption}.
Hence there is a constant $K$ such that, outside an additional event of the same order as \eqref{eq:general-boundary-failure}, $\sum_{i\in I}\frac{G_{r_i}(Y_i)}{r_i}\ls K$.
Choose $L:=\max\left\{1,\sqrt{\frac{K}{2\ln2}}\right\}$.
Then the preceding estimate gives $D_v(X)\ls2\ln2$, while the boundary event already ensures $X\pm v\in(-L,L)^m$.
Lemma~\ref{lem:curvature-criterion} therefore implies that $X$ is $v$-balanced.
Combining the two exceptional events proves \eqref{eq:general-balance-tail}.
\end{proof}

Combining Theorem~\ref{thm:general-balance} with Proposition~\ref{prop:nonuniform-triplet} gives the following theorem.

\begin{theorem}\label{thm:general-online}
Assume that $\nu$ satisfies \eqref{eq:boundary-tail-assumption}--\eqref{eq:curvature-tail-assumption} for some $q,A,a>0$ and $b\gr1$.
There are constants $L,C,c>0$ and a randomized online algorithm such that, for every fixed sequence $v_1,\ldots,v_T\in B_2^m$,
$$ \Pp\left\{ \max_{k\ls T}\norm{\sum_{t=1}^k\varepsilon_t v_t}_\infty>6L \right\} \ls C\sum_{t=1}^T \exp\!\left[-\frac{c d_t}{\ell_q(d_t)}\right], $$
where $d_t=d(v_t)$ and a zero-vector summand is interpreted as zero.
\end{theorem}

Lemma~\ref{lem:bump-input} verifies the two assumptions \eqref{eq:boundary-tail-assumption}--\eqref{eq:curvature-tail-assumption} for the density \eqref{eq:bump-density}.

\begin{proof}
Apply Theorem~\ref{thm:general-balance} to each $v_t$ and insert the resulting values of $\theta_t$ into Proposition~\ref{prop:nonuniform-triplet} with $K=L B_\infty^m$.
Since $\norm{x}_{L B_\infty^m}=L^{-1}\norm{x}_\infty$, the conclusion has the displayed form.
\end{proof}

\begin{proof}[Proof of Theorem~\ref{thm:main-online}]
Lemma~\ref{lem:bump-input}, together with Lemma~\ref{lem:calculus}, shows that \eqref{eq:boundary-tail-assumption}--\eqref{eq:curvature-tail-assumption} hold with $q=q_\beta$.
Apply Theorem~\ref{thm:general-online}.
\end{proof}

%%%%%%%%%%%%%%%%%%%%%%%%%%%%%%%%%%%%%%%%%%%%%%%%%%%%%%%%%%%%%%%%%%%%%%%%%%%%%%%%%%%%%%%%%%%%%%%%%%%%%%%%%%%%%%%%%%%%%%%%%%%%%%%%%%%%%%
\section{Lower bounds}\label{sec:offline-consequence}
%%%%%%%%%%%%%%%%%%%%%%%%%%%%%%%%%%%%%%%%%%%%%%%%%%%%%%%%%%%%%%%%%%%%%%%%%%%%%%%%%%%%%%%%%%%%%%%%%%%%%%%%%%%%%%%%%%%%%%%%%%%%%%%%%%%%%%

We shall use the two-dimensional oblivious lower bound of Kulkarni, Reis and Rothvoss~\cite[Theorem~4]{KRR} in the following form.

\begin{theorem}\label{thm:KRR-lower}
There are absolute constants $c,c'>0$ such that, for every $T$, there is an oblivious distribution over unit vectors $u_1,\ldots,u_T\in\R^2$ with the property that every randomized online signing algorithm satisfies
$$ \Pp\left\{ \max_{k\ls T}\norm{\sum_{t=1}^k\varepsilon_tu_t}_\infty \gr c\sqrt{\ln T} \right\} \gr1-2^{-T^{c'}}. $$
\end{theorem}

By Theorem~\ref{thm:KRR-lower}, we obtain the following lower bound.

\begin{proposition}\label{prop:flattened-lower}
There are absolute constants $c,c'>0$ with the following property.
For every integer $d\gr1$ and every sufficiently large $T$, against every randomized online signing algorithm there is an oblivious distribution over sequences $v_1,\ldots,v_T\in\R^{2d}$ satisfying $\norm{v_t}_2=1$ and $\norm{v_t}_\infty\ls d^{-1/2},$ such that, with probability at least $1-2^{-T^{c'}}$,
$$ \max_{k\ls T}\norm{\sum_{t=1}^k\varepsilon_t v_t}_\infty \gr c\sqrt{\frac{\ln T}{d}}. $$
Consequently, no online algorithm can guarantee a universal constant prefix discrepancy throughout the regime $d=o(\ln T)$.
\end{proposition}

For $u=(u_1,u_2)\in\R^2$ and $w\in\R^d$, set $u\otimes w:=(u_1w,u_2w)\in\R^{2d}$.
We shall use the elementary identities $\norm{u\otimes w}_2=\norm u_2\norm w_2$ and $\norm{u\otimes w}_\infty=\norm u_\infty\norm w_\infty$.

\begin{proof}[Proof of Proposition~\ref{prop:flattened-lower}]
Apply Theorem~\ref{thm:KRR-lower}, and let $u_1,\ldots,u_T\in\R^2$ denote the resulting unit vectors.
Let $w:=d^{-1/2}(1,\ldots,1)\in\R^d$ and define $v_t:=u_t\otimes w\in\R^{2d}$.
Then $\norm{v_t}_2=\norm{u_t}_2\norm{w}_2=1$ and $\norm{v_t}_\infty=\norm{u_t}_\infty\norm{w}_\infty\ls d^{-1/2}$.
Any online algorithm for the sequence $(v_t)$ induces, by feeding it $u_t\otimes w$ and using the same signs, an online algorithm for $(u_t)$.
Moreover, for every prefix,
$$ \norm{\sum_{t=1}^k\varepsilon_t v_t}_\infty =\norm{\left(\sum_{t=1}^k\varepsilon_t u_t\right)\otimes w}_\infty =\frac1{\sqrt d} \norm{\sum_{t=1}^k\varepsilon_t u_t}_\infty. $$
The claimed lower bound follows.
\end{proof}

\begin{proposition}\label{prop:compact-barrier}
Let $\varphi:[y_0,1)\to[e,\infty)$ be a $C^2$ increasing convex function such that $\varphi(y)\to\infty$ as $y\uparrow1$.
Let $L:[e,\infty)\to(0,\infty)$ be nondecreasing.
If, for all $y$ sufficiently close to $1$,
\begin{equation}\label{eq:pointwise-curvature-control}
\varphi''(y)\ls C\varphi(y)L(\varphi(y)),
\end{equation}
then
\begin{equation}\label{eq:barrier-integral}
\int^\infty\frac{du}{u\sqrt{L(u)}}<\infty.
\end{equation}
In particular, if $L(u)=\ln^q u$, then necessarily $q>2$.
\end{proposition}

\begin{proof}
After increasing $y_0$, we may assume $\varphi'(y)>0$.
Multiplying \eqref{eq:pointwise-curvature-control} by $\varphi'(y)$, integrating and then changing variables $u=\varphi(y)$ give
$$ \begin{aligned} \varphi'(y)^2 &\ls C_0+C_1\int_{\varphi(y_0)}^{\varphi(y)}uL(u)\,du\ls C_2\varphi(y)^2L(\varphi(y)),\\ \varphi'(y)&\ls C_3\varphi(y)\sqrt{L(\varphi(y))},\\ 1-y_0&=\int_{\varphi(y_0)}^\infty\frac{du}{\varphi'(\varphi^{-1}(u))}\gr c\int_{\varphi(y_0)}^\infty\frac{du}{u\sqrt{L(u)}}. \end{aligned} $$
The first line uses the monotonicity of $L$.
Thus \eqref{eq:barrier-integral} is necessary.
For $L(u)=\ln^q u$, the integral converges exactly when $q>2$.
\end{proof}

\begin{proposition}\label{prop:finite-scale-barrier}
Let $A\gr1$.
There is $c_A>0$ with the following property.
Let $\varphi:(-1,1)\to[1,\infty)$ be even, convex and $C^2$, assume that $\varphi(0)\ls2$, and let $d\gr4$.
If, for some $\Lambda\gr1$,
\begin{equation}\label{eq:finite-scale-curvature-control}
\varphi''(y)\ls A\Lambda\varphi(y) \qquad\hbox{whenever }\varphi(y)\ls d,
\end{equation}
and $\varphi$ attains the value $d$ on $(0,1)$, then $\Lambda\gr c_A\ln^2 d$.
\end{proposition}

\begin{proof}
Let $y_d\in(0,1)$ be the first point with $\varphi(y_d)=d$.
Since $\varphi$ is even and convex, $\varphi'(0)=0$ and $\varphi'(y)\gr0$ for $0\ls y\ls y_d$.
Multiplying \eqref{eq:finite-scale-curvature-control} by $\varphi'(y)$ and integrating twice gives
$$ \begin{aligned} \varphi'(y)^2&\ls A\Lambda\bigl(\varphi(y)^2-\varphi(0)^2\bigr)\ls A\Lambda\varphi(y)^2,\\ \frac{d}{dy}\ln\varphi(y)&\ls\sqrt{A\Lambda}\qquad(0\ls y\ls y_d),\\ \ln\frac{d}{\varphi(0)}&\ls\sqrt{A\Lambda}\,y_d\ls\sqrt{A\Lambda}. \end{aligned} $$
Since $\varphi(0)\ls2$ and $d\gr4$, this implies $\Lambda\gr c_A\ln^2 d$, as claimed.
\end{proof}

Proposition~\ref{prop:flattened-lower} shows that constant prefix discrepancy cannot hold uniformly when $d=o(\ln T)$.
Proposition~\ref{prop:compact-barrier} shows that a fixed compact potential satisfying the pointwise curvature estimate \eqref{eq:pointwise-curvature-control} must have logarithmic exponent strictly larger than $2$.
Proposition~\ref{prop:finite-scale-barrier} gives the corresponding lower bound $\Lambda\gr c_A\ln^2 d$ when the potential is allowed to depend on $d$.
The construction in Section~\ref{sec:critical-bump} has this order up to absolute constants.

%%%%%%%%%%%%%%%%%%%%%%%%%%%%%%%%%%%%%%%%%%%%%%%%%%%%%%%%%%%%%%%%%%%%%%%%%%%%%%%%%%%%%%%%%%%%%%%%%%%%%%%%%%%%%%%%%%%%%%%%%%%%%%%%%%%%%%
\section{A geometric extension}\label{app:geometric}
%%%%%%%%%%%%%%%%%%%%%%%%%%%%%%%%%%%%%%%%%%%%%%%%%%%%%%%%%%%%%%%%%%%%%%%%%%%%%%%%%%%%%%%%%%%%%%%%%%%%%%%%%%%%%%%%%%%%%%%%%%%%%%%%%%%%%%

For a symmetric convex body $K\subset\R^m$ we write $\norm{x}_K:=\inf\{t>0:x\in tK\}$.

The product construction of the preceding sections exploits the Euclidean condition on the arrivals and gives the stronger coordinate estimate of Theorem~\ref{thm:main-online}.
We now consider a different formulation in which both the arrivals and the discrepancy are measured in the norm of a symmetric target body $K$.
When this geometry is encoded by a norm satisfying the quadratic estimate below, a simpler compact invariant measure gives a dimension-dependent sufficient condition for constant prefix discrepancy.
The proof uses an elementary concentration estimate for convex potentials.

\begin{theorem}\label{thm:smooth-body}
There is an absolute constant $C>0$ with the following property.
Let $K\subset\R^m$ be a symmetric convex body and suppose that a norm $\norm{\cdot}_*$ satisfies $\norm{x}_K\ls\norm{x}_*\ls A\norm{x}_K$ for every $x\in\R^m$ and some $A\gr1$, and
\begin{equation}\label{eq:quadratic-smoothness-intro}
\frac{\norm{x+y}_*^2+\norm{x-y}_*^2}{2} \ls \norm{x}_*^2+\sigma\norm{y}_*^2 \qquad(x,y\in\R^m)
\end{equation}
for some $\sigma\gr1$.
Let $v_1,\ldots,v_T\in\R^m$ be fixed in advance and set $d_{K,t}=\norm{v_t}_K^{-2}$, with $d_{K,t}=\infty$ when $v_t=0$.
For every $\alpha\gr1$ there is a randomized online signing such that
\begin{equation}\label{eq:smooth-body-heterogeneous}
\Pp\left\{ \max_{k\ls T}\norm{\sum_{t=1}^k\varepsilon_t v_t}_K>6 \right\} \ls 3\sum_{t=1}^T \min\left\{1,\,2^m\left(\frac{C A^2\sigma\alpha}{d_{K,t}}\right)^{\alpha/4}\right\}.
\end{equation}
In particular, if $d_{K,t}\gr d$ for all $t$ and $d\gr C A^2\sigma$, then
\begin{equation}\label{eq:smooth-body-uniform-tail}
\Pp\left\{ \max_{k\ls T}\norm{\sum_{t=1}^k\varepsilon_t v_t}_K>6 \right\} \ls C T\exp\left(Cm-\frac{c d}{A^2\sigma}\right).
\end{equation}
Consequently, for every $0<\varepsilon<1$, the sufficient condition
\begin{equation}\label{eq:smooth-body-threshold}
d\gr C A^2\sigma\left(m+\ln\frac{3T}{\varepsilon}\right)
\end{equation}
implies that the probability in \eqref{eq:smooth-body-uniform-tail} is at most $\varepsilon$.
\end{theorem}

\begin{lemma}\label{lem:convex-potential-tail}
Let $W:\R^m\to[0,\infty]$ be convex, let $W(0)=0$, and assume that $0<Z:=\int_{\R^m}e^{-W(x)}\,dx<\infty$.
If $X$ has density $Z^{-1}e^{-W}$, then, for every $s\gr0$,
\begin{equation}\label{eq:convex-potential-tail}
\Pp\{W(X)>2m\ln2+2s\}\ls e^{-s}.
\end{equation}
\end{lemma}

\begin{proof}
For $0<t<1$ put $Z_t:=\int_{\R^m}e^{-tW(x)}\,dx$.
After the change of variables $y=tx$, convexity, $W(0)=0$ and Markov's inequality give
$$\begin{aligned} Z_t&=t^{-m}\int_{\R^m}e^{-tW(y/t)}\,dy\ls t^{-m}\int_{\R^m}e^{-W(y)}\,dy=t^{-m}Z,\\ \E e^{\theta W(X)}&=\frac{Z_{1-\theta}}{Z}\ls(1-\theta)^{-m}\qquad(0<\theta<1),\\ \Pp\{W(X)>u\}&\ls2^m e^{-u/2}. \end{aligned} $$
Here the first line uses $W(y)\ls tW(y/t)$, and the last follows by taking $\theta=1/2$; it gives \eqref{eq:convex-potential-tail}.
\end{proof}

\begin{proof}[Proof of Theorem~\ref{thm:smooth-body}]
Fix $\alpha\gr1$ and let $B_*:=\{x\in\R^m:\norm{x}_*<1\}$.
On $B_*$ consider the probability density
\begin{equation}\label{eq:smooth-density}
f_\alpha(x)=Z_\alpha^{-1}(1-\norm{x}_*^2)^\alpha\1_{B_*}(x).
\end{equation}
Its potential is $U(x)=-\alpha\ln(1-\norm{x}_*^2)$ for $x\in B_*$.
The function $U$, extended by $+\infty$ outside $B_*$, is convex and satisfies $U(0)=0$.

Fix one arrival $v$, put $\delta_*:=\norm{v}_*$ and $s(x):=1-\norm{x}_*^2$, and assume first that
\begin{equation}\label{eq:smooth-rho-small}
\rho:=C_0\sqrt{\alpha\sigma}\,\delta_*\ls\frac14,
\end{equation}
where $C_0$ is a sufficiently large absolute constant.
We claim that every $x\in B_*$ satisfying
\begin{equation}\label{eq:smooth-local-gap}
s(x)\gr\rho
\end{equation}
is $v$-balanced.

Put $F(z)=\norm{z}_*^2$ and $\phi(u)=-\alpha\ln(1-u)$.
From \eqref{eq:smooth-local-gap} and \eqref{eq:smooth-rho-small}, after increasing $C_0$ we have $\delta_*\ls s(x)/8$.
Hence $\norm{x+tv}_*^2\ls1-s(x)/2$ for every $|t|\ls1$, so the whole segment $x+[-v,v]$ lies in $B_*$.
Write $\Delta_\pm:=F(x\pm v)-F(x)$.
The smoothness assumption \eqref{eq:quadratic-smoothness-intro} and the triangle inequality give
$$ \Delta_++\Delta_-\ls2\sigma\delta_*^2, \qquad |\Delta_\pm|\ls\delta_*\bigl(\norm{x\pm v}_*+\norm{x}_*\bigr)\ls3\delta_*. $$
On the interval between $F(x)$ and $F(x\pm v)$ we have $\phi'(u)\ls2\alpha/s(x)$ and $\phi''(u)\ls4\alpha/s(x)^2$.
Taylor's formula therefore gives
$$ D_v(x) =U(x+v)+U(x-v)-2U(x) \ls C\alpha\sigma\frac{\delta_*^2}{s(x)^2}. $$
Choosing $C_0$ sufficiently large, \eqref{eq:smooth-local-gap} implies that the last expression is at most $2\ln2$.
Lemma~\ref{lem:curvature-criterion} proves the claim.

Let $X$ have density \eqref{eq:smooth-density}.
By Lemma~\ref{lem:convex-potential-tail}, $\Pp\{s(X)<r\}=\Pp\{U(X)>\alpha\ln(1/r)\}\ls2^m r^{\alpha/2}$ for every $0<r<1$.
If \eqref{eq:smooth-rho-small} holds, the claim and the preceding estimate yield $\Pp\{X\hbox{ is not }v\hbox{-balanced}\}\ls2^m\rho^{\alpha/2}$.
If \eqref{eq:smooth-rho-small} fails, we use the trivial bound $1$.
Since $\delta_*\ls A\norm{v}_K$, we conclude that
$$ \Pp\{X\hbox{ is not }v\hbox{-balanced}\}\ls\min\left\{1,\,2^m\left(C A^2\sigma\alpha\norm{v}_K^2\right)^{\alpha/4}\right\}. $$
Applying Proposition~\ref{prop:nonuniform-triplet} with $K$ replaced by the unit ball $\overline B_*$ of $\norm{\cdot}_*$ and using $\norm{x}_K\ls\norm{x}_*$ proves \eqref{eq:smooth-body-heterogeneous}.

Assume now that $d_{K,t}\gr d$ for every $t$ and $d\gr C A^2\sigma$.
Choose $\alpha:=c_0\frac{d}{A^2\sigma}$, where $c_0>0$ is a sufficiently small absolute constant.
Then $\alpha\gr1$, after adjusting the lower bound on $d$, and the quantity inside the power in \eqref{eq:smooth-body-heterogeneous} is bounded by an absolute constant strictly smaller than $1$.
Hence
$$ 2^m\left(\frac{C A^2\sigma\alpha}{d}\right)^{\alpha/4} \ls \exp\left(Cm-\frac{c d}{A^2\sigma}\right). $$
This proves \eqref{eq:smooth-body-uniform-tail}.
The last assertion follows by increasing the absolute constant in \eqref{eq:smooth-body-threshold}.
\end{proof}

For a polytope given as an intersection of slabs we obtain the following consequence.

\begin{corollary}\label{cor:slab-body}
There is an absolute constant $C>0$ such that the following holds.
Let $K=\bigcap_{i=1}^N\{x\in\R^m:|\langle a_i,x\rangle|\ls1\}$ be bounded and symmetric, let $0<\varepsilon<1$, and let $v_1,\ldots,v_T$ be fixed in advance.
If $\norm{v_t}_K\ls d^{-1/2}$ for every $1\ls t\ls T$ and
\begin{equation}\label{eq:slab-threshold}
d\gr C\left(m+\ln\frac{3T}{\varepsilon}\right)\ln(eN),
\end{equation}
then there is a randomized online signing such that
$$ \Pp\left\{ \max_{k\ls T}\norm{\sum_{t=1}^k\varepsilon_t v_t}_K>6 \right\} \ls\varepsilon. $$
\end{corollary}

\begin{proof}
Since $K$ is bounded, the functionals $a_1,\ldots,a_N$ span $\R^m$.
Put $p:=\max\{2,\ln N\}$ and define
\begin{equation}\label{eq:slab-p-norm}
\norm{x}_*:=\left(\sum_{i=1}^N|\langle a_i,x\rangle|^p\right)^{1/p}.
\end{equation}
We shall use the following classical inequality for $\ell_p$: for every $p\gr2$ and $x,y\in\R^N$,
\begin{equation}\label{eq:lp-quadratic-smoothness}
\frac{\norm{x+y}_p^2+\norm{x-y}_p^2}{2} \ls \norm{x}_p^2+(p-1)\norm{y}_p^2.
\end{equation}
This is the commutative case of the Ball--Carlen--Lieb inequality~\cite{BallCarlenLieb}.
Then
$$ \norm{x}_K =\max_{i\ls N}|\langle a_i,x\rangle| \ls\norm{x}_* \ls N^{1/p}\norm{x}_K \ls e\norm{x}_K. $$
By \eqref{eq:lp-quadratic-smoothness}, Theorem~\ref{thm:smooth-body} applies with $A\ls e$ and $\sigma=p-1\ls C\ln(eN)$, which gives \eqref{eq:slab-threshold}.
\end{proof}

The dependence on the displayed number of slabs can be replaced by the size of any fixed-factor slab approximation.
Indeed, suppose that $K$ is a symmetric convex body and $K\subseteq P\subseteq\lambda K$, where $P=\bigcap_{i=1}^N\{x:|\langle a_i,x\rangle|\ls1\}$ and $\lambda\gr1$.
If $\norm{v_t}_K\ls d^{-1/2}$, then $\norm{v_t}_P\ls d^{-1/2}$.
Corollary~\ref{cor:slab-body} therefore gives, under \eqref{eq:slab-threshold},
$$ \Pp\left\{ \max_{k\ls T}\norm{\sum_{t=1}^k\varepsilon_t v_t}_K>6\lambda \right\} \ls\varepsilon. $$
Thus the relevant combinatorial parameter may be taken from a convenient constant-factor polyhedral approximation rather than from a particular redundant presentation of $K$.

The smooth-norm formulation is often more economical.
For $K=B_2^m$ one may take $\norm{\cdot}_*=\norm{\cdot}_2$, so $A=\sigma=1$ and Theorem~\ref{thm:smooth-body} gives the scale $d\gr C(m+\ln(T/\varepsilon))$.
For $K=B_\infty^m$, taking $p=\max\{2,\ln m\}$ in \eqref{eq:slab-p-norm} gives $A\ls e$ and $\sigma\ls C\ln(em)$.
More generally, the same argument applies whenever $K$ is uniformly equivalent to the unit ball of a norm satisfying \eqref{eq:quadratic-smoothness-intro}, with the corresponding value of $\sigma$.

%%%%%%%%%%% End of paper body %%%%%%%%%%%%%%%%%%%%%%%%%%%%%%%
\bigskip

\noindent {\bf Declaration of AI-assisted technologies in the manuscript preparation process.} 
During the preparation of this work the author used ChatGPT (OpenAI) to assist with the organization, drafting and editing of parts of the manuscript. The mathematical arguments, statements, references and final text were reviewed and edited by the author who takes full responsibility for the content of the paper.

\medskip

\noindent {\bf Acknowledgement.} 
The author acknowledges support by a PhD scholarship from the National Technical University of Athens.
The author thanks Apostolos Giannopoulos for useful discussions.

\bigskip

\noindent {\bf Keywords:} discrepancy theory, online vector balancing, Koml\'os problem, convex bodies, Metropolis walk.

\smallskip

\noindent {\bf 2020 Mathematics Subject Classification:} Primary 11K38; Secondary 46B20, 60J05, 68W27.

\medskip 

\noindent \textsc{Antonios \ Hmadi}: School of Applied Mathematical and Physical Sciences, National Technical University of Athens, Department of Mathematics, Zografou Campus, GR-157 80, Athens, Greece.

\smallskip

\noindent \textit{E-mail:} \texttt{ahmadi@mail.ntua.gr}


\footnotesize
\begin{thebibliography}{100}
\footnotesize

\bibitem{AdenAli}
\textrm{I.~Aden-Ali},
\textit{Optimal online discrepancy minimization in linear time},
arXiv:2607.04388, 2026.

\bibitem{ATThreshold}
\textrm{D.~J. Altschuler and K.~Tikhomirov},
\textit{A threshold for online balancing of sparse i.i.d. vectors},
arXiv:2509.02432, 2025.

\bibitem{AT}
\textrm{D.~J. Altschuler and K.~Tikhomirov},
\textit{Online Beck--Fiala down to logarithmic sparsity},
arXiv:2607.14238, 2026.

\bibitem{ALS}
\textrm{R.~Alweiss, Y.~P. Liu, and M.~Sawhney},
\textit{Discrepancy minimization via a self-balancing walk},
Proceedings of the 53rd Annual ACM Symposium on Theory of Computing,
14--20, 2021.

\bibitem{BallCarlenLieb}
\textrm{K.~Ball, E.~A. Carlen, and E.~H. Lieb},
\textit{Sharp uniform convexity and smoothness inequalities for trace norms},
Invent. Math. \textbf{115} (1994), 463--482.

\bibitem{Banaszczyk}
\textrm{W.~Banaszczyk},
\textit{Balancing vectors and Gaussian measures of $n$-dimensional convex bodies},
Random Structures \& Algorithms \textbf{12} (1998), no.~4, 351--360.

\bibitem{BansalSurvey}
\textrm{N.~Bansal},
\textit{Discrepancy theory and related algorithms},
in \textit{International Congress of Mathematicians 2022, Vol.~7},
EMS Press, Berlin, 2023, 5178--5210.

\bibitem{BansalDadushGarg}
\textrm{N.~Bansal, D.~Dadush, and S.~Garg},
\textit{An algorithm for Koml\'os conjecture matching Banaszczyk's bound},
SIAM Journal on Computing \textbf{48} (2019), no.~2, 534--553.

\bibitem{GramSchmidtWalk}
\textrm{N.~Bansal, D.~Dadush, S.~Garg, and S.~Lovett},
\textit{The Gram--Schmidt walk: a cure for the Banaszczyk blues},
Theory of Computing \textbf{15} (2019), no.~21, 1--27.

\bibitem{BansalJiang}
\textrm{N.~Bansal and H.~Jiang},
\textit{Decoupling via affine spectral-independence: Beck--Fiala and Koml\'os bounds beyond Banaszczyk},
Proceedings of the 58th Annual ACM Symposium on Theory of Computing,
432--442, 2026; arXiv:2508.03961.

\bibitem{BJMSS}
\textrm{N.~Bansal, H.~Jiang, R.~Meka, S.~Singla, and M.~Sinha},
\textit{Online discrepancy minimization for stochastic arrivals},
in \textit{Proceedings of the 2021 ACM--SIAM Symposium on Discrete Algorithms},
2842--2861, 2021.

\bibitem{BJSS}
\textrm{N.~Bansal, H.~Jiang, S.~Singla, and M.~Sinha},
\textit{Online vector balancing and geometric discrepancy},
Proceedings of the 52nd Annual ACM Symposium on Theory of Computing,
1139--1152, 2020.

\bibitem{BansalSpencer}
\textrm{N.~Bansal and J.~H. Spencer},
\textit{On-line balancing of random inputs},
Random Structures \& Algorithms \textbf{57} (2020), no.~4, 879--891.

\bibitem{BeckFiala}
\textrm{J.~Beck and T.~Fiala},
\textit{``Integer-making'' theorems},
Discrete Applied Mathematics \textbf{3} (1981), no.~1, 1--8.

\bibitem{KRR}
\textrm{J.~Kulkarni, V.~Reis, and T.~Rothvoss},
\textit{Optimal online discrepancy minimization},
Proceedings of the 56th Annual ACM Symposium on Theory of Computing,
1832--1840, 2024; arXiv:2308.01406.

\bibitem{LSS}
\textrm{Y.~P. Liu, A.~Sah, and M.~Sawhney},
\textit{A Gaussian fixed point random walk},
in \textit{13th Innovations in Theoretical Computer Science Conference},
LIPIcs \textbf{215} (2022), 101:1--101:10.

\bibitem{LovettMeka}
\textrm{S.~Lovett and R.~Meka},
\textit{Constructive discrepancy minimization by walking on the edges},
SIAM Journal on Computing \textbf{44} (2015), no.~5, 1573--1582.

\bibitem{SmirnovVershynin}
\textrm{G.~Smirnov and R.~Vershynin},
\textit{Discrepancy and Fisher information},
arXiv:2605.13107, 2026.

\end{thebibliography}
\end{document}